\documentclass[11pt,oneside,reqno]{amsart}
\usepackage{amsaddr} % puts the addresses on the title page
\usepackage{eurosym}
\usepackage{amsmath}
\usepackage[left=2.5cm, right=2.5cm, top=2.5cm, bottom=2.5cm]{geometry}
\usepackage{amsfonts}
\usepackage{amssymb}
\usepackage{bigints}
\usepackage{lmodern}
\usepackage{enumitem}
\usepackage{graphicx}
\usepackage{hyperref}
\usepackage{caption}
\usepackage{float}% Required for the [H] placement option
\usepackage{tabularx}
\usepackage{url}
\usepackage{xfrac}
\usepackage{amsthm}
\usepackage{amscd}
\usepackage{amsfonts}
\usepackage{epsfig}
\usepackage{graphicx,psfrag}
\usepackage{mathtools}
\usepackage{color}
\usepackage{tikz}
\usepackage{changepage}

\usepackage{listings}
\usepackage{xcolor}

\usepackage{fancyvrb}

\lstdefinelanguage{Mathematica}{
  morekeywords={
    Module, If, Then, Else, Return, Print, Set, SetDelayed, Range,
    IntegerQ, Select, Flatten, Import, FileNames, MatchQ, Min, Max,
    ToExpression, Length, Complement, ParallelMap
  },
  sensitive=true,
  morecomment=[l](*),
  morestring=[b]",
}

\newtheorem {theorem} {Theorem} [section]

\newtheorem {corollary} [theorem] {Corollary}
\newtheorem {conjecture}{Conjecture}
\newtheorem {remark}{Remark}
\newtheorem {lemma}{Lemma} [section]

\title[The generalized Erd\H{o}s--Straus equation]{Almost a Complete Proof of the Generalized Erd\H{o}s--Straus Conjecture: {5}/{a} = {1}/{b} + {1}/{c} + {1}/{d}}
\date{\today}

\begin{document}
\author[B. Ghermoul]{Bilal Ghermoul$^{(1)}$}
\address{$^{(1)}$ Department of Mathematics, Faculty of Mathematics and Computer Science,\\ University Mohamed El Bachir El Ibrahimi  of Bordj Bou Arreridj,\\
     El-Anasser 34030, Algeria}
\email{bilal.ghermoul@univ-bba.dz}


\maketitle

\begin{abstract}
The generalized Erd\H{o}s--Straus conjecture, proposed by Wac\l{}aw Sierpi\'{n}ski in 1956, asks whether the Diophantine equation
\[
\frac{5}{a} = \frac{1}{b} + \frac{1}{c} + \frac{1}{d}
\]
admits positive integer solutions $b,c,d \in \mathbb{N}$ for every integer $a \ge 2$.
In this work we present explicit solutions for all integers $a \ge 2$. 
We begin with the simplest known cases where $a \equiv i \pmod{5}$ for $i \in \{0,2,3,4\}$, providing direct decompositions. 
The remaining open case, $a \equiv 1 \pmod{5}$, is addressed for $a = 5q + 1$ with $q \not\equiv 0 \pmod{252}$, where we give explicit decompositions, often with $q$ expressed as three-variable polynomials.
For $q \equiv 0 \pmod{252}$, we conjecture that a specific polynomial 
$p_{1}(x,y,z)=z (x (5 y-1)-y)-x,~ x,y,z \in \mathbb{N}^*$, which exactly satisfies the generalized Erd\H{o}s--Straus 
equation, generates all such multiples of $252$. This conjecture has been 
verified computationally for $5q+1$ up to approximately $10^{10}$, and the corresponding \textit{Mathematica} implementation is included.
\end{abstract}

\date{}

\section{Introduction}

The well-known Erdős--Straus conjecture in number theory, proposed by Paul Erdős and Ernst G. Straus in 1948 \cite{elsholtz2001sums,kotsireas1999erdos}, asserts that every positive integer greater than or equal to 2 can be represented as the sum of three unit fractions. This conjecture has attracted considerable interest in the mathematical community due to its apparent simplicity and the challenging nature of its proof. Despite its straightforward statement, the conjecture has remained unresolved for many years, capturing the curiosity of mathematicians globally.

Over time, various mathematicians have examined different facets of the Erdős--Straus conjecture, resulting in a substantial body of literature. Notable contributions include works by L. Bernstein \cite{bernstein1962losung}, Nathanson \cite{nathanson1994additive,nathanson2002proving}, Konyagin and Shorey \cite{konyagin1999number}, Ahlgren, Ono, and Penniston \cite{ahlgren2008zeta}, Vaserstein \cite{vaserstein2008geometric}, Helfgott and Harcos \cite{helfgott2013erdos}, Karasev \cite{karasev2015erdos}, Farnsworth \cite{farnsworth2018non}, Crawford \cite{crawford2019number}, Giovanni, Gallipoli, Gionfriddo \cite{giovanni2019historical}, Elsholtz and Tao \cite{elsholtz2013counting}, Ghanouchi \cite{ghanouchi2012analytic,ghanouchi2015erdos}, Mordell \cite{mordell1969diophantine}, Subburam and Togbé \cite{Subburam-2016}, Negash \cite{negash2018solutions}, Oblàth \cite{oblath1950equation}, Rosati \cite{rosati1954sull}, Sander \cite{sander1994half}, Vaughan \cite{vaughan1970problem}, Yamamoto \cite{yamamoto1965diophantine}, and many others. The conjecture's validity for integers up to $a \leq 10^{14}$ and $a \leq 10^{17}$ was verified by Swett \cite{swett1999erdos} and Salez \cite{salez2014erdos}, respectively.

A generalized version of the Erd\H{o}s--Straus conjecture states that, for
 any positive \(n\), all but finitely many fractions \({n}/{a}\)
 can be expressed as a sum of three positive unit fractions. The conjecture
 for fractions \({5}/{a}\) was made by Wac\l{}aw Sierpi\'{n}ski in 1956,
and the full conjecture was later attributed to Sierpi\'{n}ski's student, Andrzej Schinzel \cite{sierpinski1956,vaughan1970problem}.

In \textit{Unsolved Problems in Number Theory}, Guy~\cite{Guy} presents the conjecture as a long-standing open problem, highlighting its enduring importance in the field of number theory.

In this paper, we explore a generalization of the Erdős--Straus conjecture and provide an explicit solution demonstrating its validity for all positive integers 
\( a \geq 2 \).

We begin by introducing the Generalized Erd\H{o}s--Straus Conjecture:

\begin{conjecture}[Generalized Erd\H{o}s--Straus Conjecture]
For every integer \( n \geq 2 \) and every integer \( r \geq 2 \), the Diophantine equation
\begin{equation}
\label{classicESC}
\frac{r}{a} = \frac{1}{b} + \frac{1}{c} + \frac{1}{d}
\end{equation}
has a solution in positive integers \( b, c, d \).
\end{conjecture}

The classical Erd\H{o}s--Straus Conjecture corresponds to the case \( r = 4 \). 
In the present work, we focus on the case \( r = 5 \) and begin by establishing 
some simple results for the general case with arbitrary \( r \geq 2 \).

\section{Main results}
\label{sec1}
Our main result is summarized in the following.

First, we consider a generatlized Erd\H{o}s--Straus equation as follows
\begin{equation}\label{GESE}
\frac{r}{q r+t}=\frac{1}{(n+q) (q r+t)} + \frac{1}{m (q+s)} + \frac{1}{q+s},
\end{equation}
where $r,q\geq 2$ and $m,q+s,q+n \geq 1$.

\begin{lemma}
\label{thm1}
The generalized Erd\H{o}s--Straus Equation~\eqref{GESE} holds for every \( q \in \mathbb{N}^* \) that satisfies one of the following conditions:
\begin{itemize}
  \item[\hspace{0.5em}(a)\hspace{0.5em}]  There exists an integer $\kappa \geq 1$, such that
  \begin{equation}
  \label{conj1}
  \frac{(m+1) (q r+t)}{r (n+q)-1}=\frac{m (q+s)}{n+q}=\kappa.
  \end{equation}
  
  \item[\hspace{0.5em}(b)\hspace{0.5em}] There exists an integer $\kappa \geq 1$ such that

\begin{equation}
  \label{conj2}
 r m \mid (1+m) (q r+t)+\kappa \quad \text{and} \quad r\kappa \mid (1+m) (q r+t)+\kappa.
\end{equation}  
  
  \item[\hspace{0.5em}(c)\hspace{0.5em}]  There exist $\kappa,z \geq 1$ such that
\begin{equation}
\label{conj3}
 q=\kappa z-s \quad \text{and} \quad \frac{\kappa  (r z+1)}{r s-t} = c \in \mathbb{N}^*.
\end{equation}
  
  \item[\hspace{0.5em}(d)\hspace{0.5em}]  There exist $\kappa,z \geq 1$ such that
\begin{equation}
\label{conj4}
 q=\beta-s \quad \text{and} \quad \frac{\beta  (\kappa +\beta  r)}{\kappa  (r s-t)} = c \in \mathbb{N}^*.
\end{equation}
\end{itemize}
\end{lemma}

\begin{proof}[{\bf Proof of Lemma~\ref{thm1}}] The theorem holds clearly in each of the four  cases; we now justify each in turn.
\begin{itemize}
  \item[\hspace{0.5em}(a)\hspace{0.5em}]  The matter is clear; it suffices to take 
  \begin{equation}
  \label{proofconj1}
  \kappa = m (r s-t)-(q r+t).
  \end{equation}
  
  \item[\hspace{0.5em}(b)\hspace{0.5em}] For this case, it suffices to consider  
\begin{equation}
\label{proofconj2}
 n = \frac{\kappa +(m+1) (q r+t)}{\kappa  r}-q \quad \text{and} \quad 
 s = \frac{\kappa +(m+1) (q r+t)}{m  r}-q.
\end{equation}   
Clearly, the values of \( s \) and \( n \) defined by \eqref{proofconj2} satisfy equation \eqref{GESE}.  
Therefore, in order for \( s \) and \( n \) to be integers, condition \eqref{conj2} must hold.
  
  \item[\hspace{0.5em}(c)\hspace{0.5em}]  Referring to equations~\eqref{conj1} and~\eqref{proofconj2}, we now solve the algebraic system
\begin{equation*}
	(m+1) (q r+t)  =\kappa  (\alpha  r-1),\quad	
	\alpha  \kappa  =\beta  m, \quad
	\beta  = q+s, \quad \text{and} \quad 
	q+s  =\kappa  z,
\end{equation*}
to obtain 
\begin{equation*}
	\beta = \kappa  z,  \quad
 	 m = \frac{\kappa +\kappa  r z}{r s-t}-1, \quad
	\alpha  = z \left(\frac{\kappa +\kappa  r z}{r s-t}-1\right) , \quad 
	\text{and} 
	\quad 
	q = \kappa  z-s.
\end{equation*}
Substituting these expressions into equation~\eqref{GESE}, we obtain
\begin{equation}
	\label{erdosh1}
\frac{r}{r (\kappa  z-s)+t}=\frac{1}{z \left(\frac{\kappa  (r z+1)}{r s-t}-1\right) (r (\kappa  z-s)+t)}+\frac{1}{\kappa  z \left(\frac{\kappa  (r z+1)}{r s-t}-1\right)}+\frac{1}{\kappa  z}.
\end{equation}
Since this equation is always true, equation~\eqref{GESE} also holds, provided that the numerators on the right-hand side are integers, which occurs precisely when condition~\eqref{conj3} is satisfied.

  \item[\hspace{0.5em}(d)\hspace{0.5em}] Also refering to equations~\eqref{conj1} and~\eqref{proofconj2}, 
  we solve the algebraic system
  \begin{equation*}
  	(m+1) (q r+t)=\kappa  (\alpha  r-1),\quad	
  	\alpha  \kappa  =\beta  m, \quad \text{and} \quad 
  	\beta  = q+s,
  \end{equation*}
  we obtain 
  \begin{equation*}
  	\alpha = \frac{\beta  (\kappa +r (\beta -s)+t)}{\kappa  (r s-t)},  \quad
  	m = \frac{\kappa +r (\beta -s)+t}{r s-t} , \quad 
  	\text{and} 
  	\quad 
  	q = \beta-s.
  \end{equation*}
  Substituting these expressions into equation~\eqref{GESE}, we obtain
  \begin{equation}
  	\label{erdosh2}
  	\frac{r}{r (\beta -s)+t}=\frac{1}{\beta }+\frac{1}{\frac{\beta}{\kappa } \left(\frac{\kappa +\beta  r}{r s-t}-1\right) (r (\beta -s)+t)}+\frac{1}{\beta  \left(\frac{\kappa +\beta  r}{r s-t}-1\right)}.
  \end{equation}
  As the equation is identically satisfied, Equation~\eqref{GESE} holds as well, 
  provided the numerators on the right-hand side are integers. This condition is met 
  exactly when condition~\eqref{conj4} is satisfied.
\end{itemize}
This completes the proof.
\end{proof}

  \begin{remark}
	It is well known that Conjecture~\ref{classicESC} remains unsolved for \( r = 4 \) and \( r = 5 \), particularly when \( a \) is a prime of the form \( 4q + 1 \) or \( 5q + 1 \), respectively. For this reason, our focus will be on proving Equation~\eqref{GESE} in the cases \( (r, t) = (4, 1) \) and \( (r, t) = (5, 1) \).
\end{remark}

Since our aim is to prove the generalized Erd\H{o}s--Straus Conjecture~\ref{classicESC} 
with \( r = 5 \), we focus instead on Equation~\eqref{GESE} with \( (r, t) = (5, 1) \). 
Before proceeding, we address the cases where \( a \neq 5q + 1 \); this is the subject 
of the following lemma.

\begin{lemma}
\label{leme}
The generalized Erdős--Straus decomposition~\eqref{classicESC}, with \(r=4,5,\) or 6
have solutions for every \(a \neq r\,k+1,~~k\in \mathbb{N}^*\), with \(\min\{b,c,d\} = \left[\frac{a}{r}\right]+1\), where \(\left[\frac{a}{r}\right]\) is the integer part of \(\frac{a}{r}\).
\end{lemma}

\begin{theorem}
	\label{thm2}
	Let \( a \geq 2 \) be an integer. Then there exist positive integers \( b, c, d \) 
	such that the generalized Erd\H{o}s--Straus decomposition 
	\begin{equation}\label{DE1}
		\frac{5}{a}=\frac{1}{b}+\frac{1}{c}+\frac{1}{d},
	\end{equation}
	holds in each of the 
	following cases for \( q \in \mathbb{Z}_{>0} \):
	
	\begin{enumerate}[label=(\arabic*), labelsep=0.6em, left=0pt, itemsep=0.75em]
		\item\label{state1} If  \( a = 5q+i,\) \(i \in \{0,2,3,4\},\)  then such a decomposition exists.
		\item\label{state2} If \( a = 5q+1\) with \(q \not\equiv 0 \pmod{12},\) then such a decomposition exists.
		\item\label{state3} If \(q = 12 u,~u\geq1 \), for \( u = 7v+i,~v\geq 0\) and \(i \in \{1,2,3,4,5,6\},\) then such a decomposition exists.
		\item\label{state4} If \( u = 7v \), then \( q = 84 v \).  
In this case, if \( v = 3w + i \) with \( i \in \{1,2\} \), then such a decomposition exists.
	\end{enumerate}
\end{theorem}

\medskip
\noindent
We note that Theorem~\eqref{thm2} does not address the case \(q \equiv 0 \pmod{252}\).
The following conjecture is intended to cover this remaining case.
\medskip

\begin{conjecture}\label{cnjctr}
Let \(p_{1}(x,y,z) \in \mathbb{Z}[x,y,z]\) be the polynomial~\eqref{p1} defined in Section~\ref{sec3} (Proof of Statement~\ref{state2} of Theorem~\ref{thm2}).
Then, for every positive integer \(\ell\), there exist \(x, y, z \in \mathbb{N}^*\) such that
\[
p_{1}(x,y,z) = 252\ell.
\]
Equivalently, \(252\mathbb{N}^* \subseteq p_{1}(\mathbb{N}^{*3})\).
\end{conjecture}
A computational verification of Conjecture~\ref{cnjctr},  
up to \(5q+1 \approx 2\times 10^{10}\) for \(q \equiv 0 \pmod{252}\),  
is provided in Appendix~\ref{additional}.

Theorem \ref{thm2} is proved in Section \ref{sec3}.

\begin{proof}[{\bf Proof of Lemma \ref{leme}}]
We start by solving equation \eqref{classicESC} with respect to \(d\) to obtain
\begin{equation}\label{deq1}
	d=\frac{a\, b\, c}{r\, b\, c-a\, (b+c)}.
\end{equation}
Since \(\frac{1}{d}=\frac{r}{a}-\frac{1}{b}-\frac{1}{c}\) then \(\frac{1}{d}<\frac{r}{a}\) which means that
\[
d > \frac{a}{r} \Rightarrow d \geq \left[ \frac{a}{r} \right] + 1.
\]
We consider the minimum value for \(d\), then we obtain the following
decompositions for \(\frac{r}{a}\) regarding \(r=4,5,6\).

\noindent {\bf (1) \underline{For \(r=4\):}} Consider \(a= 4\,k+i,~~ k\in \mathbb{N}^*\) and \(i=0,2,3\). Then we have the following:

\noindent {\bf (1.1)}  If \(a= 4\,k,~~ k\in \mathbb{N}^*\), we get the decomposition
	\begin{equation*}
		\frac{4}{4\, k}=\frac{1}{k} = \frac{1}{(k+1)^2}+\frac{1}{k\, (k+1)^2}+\frac{1}{k+1}.
	\end{equation*}
	
\noindent {\bf (1.2)}  If \(a= 4\,k+2,~~ k\in \mathbb{N}^*\), we get the decomposition
	\begin{equation*}
		\frac{4}{4\, k+2} = \frac{1}{2\, (k+1)^2}+\frac{1}{2\, (k+1)^2 (2\, k+1)}+\frac{1}{k+1}.
	\end{equation*}
	
\noindent {\bf (1.3)}  If \(a= 4\,k+3,~~ k\in \mathbb{N}^*\), we get the decomposition
	\begin{equation*}
		\frac{4}{4\, k+3} = \frac{1}{4\, (k+1)^2}+\frac{1}{4\, (k+1)^2 (4\, k+3)}+\frac{1}{k+1}.
	\end{equation*}
	
\noindent {\bf (2) \underline{For \(r=5\):}} Consider \(a= 5\,k+i,~~ k\in \mathbb{N}^*\) and \(i=0,2,3,4\). Then we have the following:

\noindent {\bf (2.1)} If \(a= 5\,k,~~ k\in \mathbb{N}^*\), we get the decomposition 
\begin{equation*}
    \frac{5}{5\, k}=\frac{1}{k} = \frac{1}{(k+1)^2}+\frac{1}{k\, (k+1)^2}+\frac{1}{k+1}.
  \end{equation*}

\noindent {\bf (2.2)} If \(a= 5\,k+2,~~ k\in \mathbb{N}^*\), we get the decomposition
  \[
  \frac{5}{5\, k+2} =
  \begin{cases}
  	\dfrac{1}{10\, q^2+7\, q+1}+\dfrac{1}{20\, q^2+14\, q+2}+\dfrac{1}{2\, q+1}, & \text{if } k=2\,q,   	\vspace{8pt}\\
  	\dfrac{1}{10\, q^2+17\, q+7}+\dfrac{1}{20\, q^2+34\, q+14}+\dfrac{1}{2\, q+2},  & \text{if } k=2\,q+1.
  \end{cases}
  \]

\noindent {\bf (2.3)} If \(a= 5\,k+3,~~ k\in \mathbb{N}^*\), we get the decomposition
    \begin{equation*}
    \frac{5}{5\, k+3} = \frac{1}{5\, k^2+8\, k+3}+\frac{1}{5\, k^2+8\, k+3} +\frac{1}{k+1}.
  \end{equation*}

\noindent {\bf (2.4)} If \(a= 5\,k+4,~~ k\in \mathbb{N}^*\), we get the decomposition
    \begin{equation*}
    \frac{5}{5\, k+4} = \frac{1}{5\, (k+1)^2 (5\, k+4)}+\frac{1}{5\, (k+1)^2} +\frac{1}{k+1}.
  \end{equation*}
  
\noindent {\bf (3) \underline{For \(r=6\):}} Consider \(a= 6\,k+i,~~ k\in \mathbb{N}^*\) and \(i=0,2,3,4,5\). Then we have the following:

\noindent {\bf (3.1)} If \(a= 6\,k,~~ k\in \mathbb{N}^*\), we get the decomposition \(\tfrac{6}{6\, k} = \tfrac{1}{k}\).
	
\noindent {\bf (3.2)} If \(a= 6\,k+2,~~ k\in \mathbb{N}^*\), we get the decomposition
	\begin{eqnarray*}
		\frac{6}{6\, k+2} &=& \frac{1}{3\, k^2+4\, k+1}+\frac{1}{3\, k^2+4\, k+1}+\frac{1}{k+1}.
	\end{eqnarray*}
	
\noindent {\bf (3.3)} If \(a= 6\,k+3,~~ k\in \mathbb{N}^*\), we get the decomposition
	\begin{equation*}
		\frac{6}{6\, k+3} = \frac{1}{2\, (k+1)^2}+\frac{1}{2\, (k+1)^2 (2\, k+1)} +\frac{1}{k+1}.
	\end{equation*}
	
\noindent {\bf (3.4)} If \(a= 6\,k+4,~~ k\in \mathbb{N}^*\), we get the decomposition
	\begin{equation*}
		\frac{6}{6\, k+4} = \frac{1}{3\, (k+1)^2}+\frac{1}{3\, (k+1)^2 (3\, k+2)} +\frac{1}{k+1}.
	\end{equation*}
	
\noindent {\bf (3.5)} If \(a= 6\,k+5,~~ k\in \mathbb{N}^*\), we get the decomposition
	\begin{equation*}
		\frac{6}{6\, k+5} = \frac{1}{6\, (k+1)^2}+\frac{1}{6\, (k+1)^2 (6\, k+5)} +\frac{1}{k+1}.
	\end{equation*}
This completes the proof.
\end{proof}

Now, let us return to the proof of Theorem \ref{thm2}.

\section{Proof of Theorem Theorem \ref{thm2}}
\label{sec3}

\begin{proof}[{\bf Proof of Statement~\ref{state1} of Theorem \ref{thm2}}]
The proof of Theorem~\ref{thm2} follows directly in the case \( a \neq 5q + 1 \), with \( q \in \mathbb{N} \), by invoking Lemma~\ref{leme}. 
\end{proof}

\begin{proof}[{\bf Proof of Statement~\ref{state2} of Theorem \ref{thm2}}]
Now, we consider the case where \( a = 5q + 1 \), with \( q \in \mathbb{N} \). 

First, observe that any natural number \( q \) can be written as
\begin{equation}
\label{zks}
q = z \kappa - s,
\end{equation}
for any $z,\kappa,s \in \mathbb{N}^*$, that is due to arbitrariness of these integers.
In addition, the generalized decomposition defined by 
Equation~\eqref{erdosh1}, when setting \( (r, t) = (5, 1) \), 
becomes:
\begin{equation}
\label{erod5}
\frac{5}{5 (\kappa  z-s)+1}=\frac{1}{\kappa  z \left(\frac{\kappa  (5 z+1)}{5 s-1}-1\right)}+\frac{1}{z \left(\frac{\kappa  (5 z+1)}{5 s-1}-1\right) (5 (\kappa  z-s)+1)}+\frac{1}{\kappa  z}.
\end{equation}
This identity holds for every \( q \) of the form~\eqref{zks}. Then, for \( q \) to satisfy the decomposition~\eqref{DE1}, it must meet one of the following conditions in full.

\noindent \textbf{Condition 01:} This is the case when 
\begin{equation}
\label{cond1}
\frac{5 z+1}{5 s-1}=c\in \mathbb{N}^* \text{  and  } \kappa  z-s=q.
\end{equation}
This means that there exists \( \gamma \in \mathbb{N}^* \) such that 
\[
z = (5\gamma - 1)s - \gamma,
\]
and therefore
\[
q = \kappa \big((5\gamma - 1)s - \gamma\big) - s.
\]
Replacing \( (s, \gamma, \kappa) \to (x, y, z) \), the expression for \( q \) 
becomes of the form:
\begin{equation}
\label{p1}
\boxed{p_1(x,y,z) = z (x (5 y-1)-y)-x, \quad x,y,z \in \mathbb{N}^*.}
\end{equation}
Here, \( p_1 \in \mathbb{Z}[x, y, z] \) is a polynomial in three variables with 
integer coefficients; that is, \( \mathbb{Z}[x, y, z] \) denotes the ring of 
polynomials in \( x, y, z \) over the integers.

Thus, equations~\eqref{DE1} is indeed satisfied for $a=5q+1$, when $q=p_1(x,y,z)$, and from Equation~\eqref{erod5} we obtain:
\begin{equation}
\label{p1Erdos}
\begin{aligned}
\frac{5}{5 p_1(x,y,z)+1} &= \frac{5}{5 (z (x (5 y-1)-y)-x)+1} \\
&= \frac{1}{(x (5 y-1)-y) ((5 y-1) z-1) (5 x ((5 y-1) z-1)-5 y z+1)} \\
 &+ \frac{1}{z (x (5 y-1)-y) ((5 y-1) z-1)}+\frac{1}{z (x (5 y-1)-y)}.
\end{aligned}
\end{equation}

\noindent {\bf C.1.1.} Polynomial $p_1$ generate the congruence \(\underline{q \equiv 2 \pmod{12}},\) with $x=1,$ $y=1$, and $z=1+4 x$, $x\geq 0$, which gives
\begin{equation}
\label{p1Case1}
\frac{5}{1+5 \underline{(2+12 x)}}=\frac{1}{3 (1+4 x) (3+16 x)}+\frac{1}{3 (3+16 x) (11+60 x)}+\frac{1}{3 (1+4 x)}.
\end{equation}

\noindent {\bf C.1.2.} Also, polynomial $p_1$ generate the congruence \(\underline{q \equiv 5 \pmod{12}},\) with $x=1,$ $y=1$, and $z=2 + 4 x$, $x\geq 0$. Then we obtain
\begin{equation}
\label{p1Case2}
\frac{5}{1+5\underline{(5+12 x)}}=\frac{1}{6 (1+2 x) (7+16 x)}+\frac{1}{6 (7+16 x) (13+30 x)}+\frac{1}{3 (2+4 x)}.
\end{equation}

\noindent {\bf C.1.3.} $p_1$ also generate the congruence \(\underline{q \equiv 8 \pmod{12}},\) with $x=1,$ $y=1$, and $z=3 + 4 x$, $x\geq 0$. Then we obtain
\begin{equation}
\label{p1Case3}
\frac{5}{1+5\underline{(8+12 x)}}=\frac{1}{3 (3+4 x) (11+16 x)}+\frac{1}{3 (11+16 x) (41+60 x)}+\frac{1}{3 (3+4 x)}.
\end{equation}

\noindent {\bf C.1.4.} $p_1$ also generate the congruence \(\underline{q \equiv 6 \pmod{12}},\) with $x=1,$ $y=2+3 x$, and $z=1$, $x\geq 0$. We obtain
\begin{equation}
\label{p1Case4}
\frac{5}{1+5 \underline{(6+12 x)}}=\frac{1}{(7+12 x) (8+15 x)}+\frac{1}{(7+12 x) (8+15 x) (31+60 x)}+\frac{1}{7+12 x}.
\end{equation}

\noindent {\bf C.1.5.} $p_1$ also generate the congruence \(\underline{q \equiv 10 \pmod{12}},\) with $x=1,$ $y=3+3 x$, and $z=1$, $x\geq 0$. We obtain
\begin{equation}
\label{p1Case5}
\frac{5}{1+5 \underline{(10+12 x)}}=\frac{1}{(11+12 x) (13+15 x)}+\frac{1}{3 (11+12 x) (13+15 x) (17+20 x)}+\frac{1}{11+12 x}.
\end{equation}

\noindent \textbf{Condition 02:} This is the case when 
\begin{equation}
\label{cond2}
\frac{\kappa  (5 z+1)}{5 s-1}=c\in \mathbb{N}^* \text{  and  } \kappa  z-s=q.
\end{equation}

First, we consider \(c = 2\), in this case \eqref{cond2} is fulfilled
by considering the following parameters:
\[
q = (1+3 c_2)+(2+5 c_2) c_1,\quad s = (2+3 c_2)+(3+5 c_2) c_1,\quad z = 1+2 c_2, 
\quad \text{and} \quad \kappa = 3+5 c_1,
\]
with \(c_1,c_2 \geq 0\), relpacing \((c_1,c_2) \to (x-1,y-1)\)
and therefore an expression for \( q \) 
becomes of the form:
\begin{equation}
	\label{p2}
	\boxed{p_2(x,y) = x (5 y-3)-2 y+1, \quad x,y,z \in \mathbb{N}^*.}
\end{equation}
Here, \( p_2 \in \mathbb{Z}[x, y] \) is a polynomial in three variables with 
integer coefficients; that is, \( \mathbb{Z}[x, y] \) denotes the ring of 
polynomials in \( x, y, z \) over the integers.

Thus, equations~\eqref{DE1} is indeed satisfied for $a=5q+1$, with $q=p_2(x,y)$, 
and from Equation~\eqref{erod5} we obtain:
\begin{equation}
	\label{p2Erdos}
	\begin{aligned}
		\frac{5}{5 p_2(x,y)+1} &=\frac{5}{(5 x-2) (5 y-3)} \\
		&= \frac{1}{(5 x-2) \left(10 y^2-11 y+3\right)}+\frac{1}{10 x y-5 x-4 y+2}+\frac{1}{10 x y-5 x-4 y+2}.
	\end{aligned}
\end{equation}

Secondly, we consider \(c = 3\), in this case \eqref{cond2} is fullfilled
by consideriong following parameters:
\[
q = (1+4 c_2)+(3+10 c_2) c_1,\quad s = (1+2 c_2)+(2+5 c_2) c_1,\quad z = 1+3 c_2, 
\quad \text{and} \quad \kappa = 2+5 c_1,
\]
with \(c_1,c_2 \geq 0\), relpacing \((c_1,c_2) \to (x-1,y-1)\)
and therefore an expression for \( q \) 
becomes of the form:
\begin{equation}
	\label{p3}
	\boxed{p_3(x,y) = x (10 y-7)-6 y+4, \quad x,y,z \in \mathbb{N}^*.}
\end{equation}
Here, \( p_3 \in \mathbb{Z}[x, y] \) is a polynomial in three variables with 
integer coefficients; that is, \( \mathbb{Z}[x, y] \) denotes the ring of 
polynomials in \( x, y, z \) over the integers.

Thus, equations~\eqref{DE1} is indeed satisfied for $a=5q+1$, with $q=p_3(x,y)$, 
and from Equation~\eqref{erod5} we obtain:
\begin{equation}
	\label{p3Erdos}
	\begin{aligned}
		\frac{5}{5 p_3(x,y)+1} &=\frac{5}{(5 x-3) (10 y-7)} \\
		&= \frac{1}{2 (5 x-3) \left(30 y^2-41 y+14\right)}+\frac{1}{30 x y-20 x-18 y+12}
		+\frac{1}{15 x y-10 x-9 y+6}.
	\end{aligned}
\end{equation}

\noindent {\bf C.2.1.} Condition~\eqref{cond2} generate the \(\underline{q \equiv 1 \pmod{12}}\), as follows
\begin{equation}
\label{p2Case1}
\frac{5}{1+5 \underline{(1+12 x)}}=\frac{1}{6+60 x}+\frac{1}{3+30 x}+\frac{1}{3+30 x}.
\end{equation}
with $c=2,$ $s=2+18 x$, $z=1$, and $\kappa=3+30 x$, $x\geq 0$.

\noindent {\bf C.2.2.} And \eqref{cond2} generate the \(\underline{q \equiv 9 \pmod{12}}\), as follows
\begin{equation}
\label{p2Case2}
\frac{5}{1+5 \underline{(9+12 x)}}=\frac{1}{46+60 x}+\frac{1}{23+30 x}+\frac{1}{23+30 x},
\end{equation}
with $c=2,$ $s=2 (7+9 x)$, $z=1$, and $\kappa=23+30 x$, $x\geq 0$.

\noindent {\bf C.2.3.} And \eqref{cond2} generate the \(\underline{q \equiv 3 \pmod{12}}\), as follows
\begin{equation}
\label{p2Case3}
\frac{5}{1+5 \underline{(3+12 x)}}=\frac{1}{16+60 x}+\frac{1}{8+30 x}+\frac{1}{8+30 x},
\end{equation}
with $c=2,$ $s=5+18 x$, $z=1$, and $\kappa=8+30 x$, $x\geq 0$.

\noindent {\bf C.2.4.} And \eqref{cond2} generate the \(\underline{q \equiv 11 \pmod{12}}\), as follows
\begin{equation}
\label{p2Case4}
\frac{5}{1+5 \underline{(11+12 x)}}=\frac{1}{56+60 x}+\frac{1}{28+30 x}+\frac{1}{28+30 x},
\end{equation}
with $c=2,$ $s=5+6 x$, $z=1$, and $\kappa=8+10 x$, $x\geq 0$.

\noindent {\bf C.2.5.} And \eqref{cond2} generate the \(\underline{q \equiv 4 \pmod{12}}\), as follows
\begin{equation}
\label{p2Case5}
\frac{5}{1+5 \underline{(4+12 x)}}=\frac{1}{6 (7+20 x)}+\frac{1}{14+40 x}+\frac{1}{7+20 x},
\end{equation}
with $c=3,$ $s=3+8 x$, $z=1$, and $\kappa=7+20 x$, $x\geq 0$.

\noindent {\bf C.2.6.} And \eqref{cond2} generate the \(\underline{q \equiv 7 \pmod{12}}\), as follows
\begin{equation}
\label{p2Case6}
\frac{5}{1+5 \underline{(7+12 x)}}=\frac{1}{8 (3+5 x)}+\frac{1}{24 (3+5 x)}+\frac{1}{4 (3+5 x)},
\end{equation}
with $c=3,$ $s=5+8 x$, $z=1$, and $\kappa=4 (3+5 x)$, $x\geq 0$.

This completes the proof of Statement~\ref{state2} of Theorem \ref{thm2}. 
\end{proof}
Based on the preceding proof, we state the following remark:
\begin{remark}
All cases~{\bf C.1.1–C.1.5} can be reduced to only two representative cases:

\noindent {\bf $\bullet$} In the first case, we substitute \( p_1(1,1,z) \) into~\eqref{p1Erdos} 
to obtain
\[
\frac{5}{5 (3 z-1)+1}=\frac{1}{3 z (4 z-1)}+\frac{1}{3 (4 z-1) (15 z-4)}+\frac{1}{3 z},
\]
Alternatively, by setting \( s = 1 \), \( c = 4 \), \( \kappa = x + 1 \), and \( z = 3 \), we get \( q = 3x + 2 \) (from \eqref{cond1}) 
in~\eqref{erod5}, which leads to a similar result.

\noindent {\bf $\bullet$} In the second case, we substitute \( p_1(1, y, 1) \) into~\eqref{p1Erdos} to obtain
\[
\frac{5}{5 (4 y-2)+1}=\frac{1}{(4 y-1) (5 y-2)}+\frac{1}{(4 y-1) (5 y-2) (20 y-9)}+\frac{1}{4 y-1}.
\]
Alternatively, by setting \( s = 1 \), \( c = 5x + 4 \), \( \kappa = 1 \), and \( z = 4x + 3 \), we get \( q = 4x + 2 \) (from \eqref{cond1}) 
in~\eqref{erod5}, again yielding a similar result.

In a similar manner, all cases~{\bf C.2.1–C.2.6} can be reduced to only two 
representative cases:

\noindent {\bf $\bullet$} In the first case, we set \( c = 2 \), \( s = 3x + 2 \), 
\( \kappa = 5x + 3 \), and \( z = 1 \). From~\eqref{cond2}, we obtain \( q = 2x + 1 \). 
Substituting into~\eqref{p1Erdos}, 
we get
\[
\frac{5}{5 (2 x+1)+1}=\frac{1}{10 x+6}+\frac{1}{5 x+3}+\frac{1}{5 x+3}.
\]
As another approach, setting \( p_2(x,1) = 2x - 1 \) results in a similar expression.

\noindent {\bf $\bullet$} In the second case, we set \( c = 3 \), \( s = 2x + 1 \), 
\( \kappa = 5x + 3 \), 
and \( z = 1 \). From~\eqref{cond2}, we obtain \( q = 3x + 1 \). Substituting 
into~\eqref{p1Erdos}, we get
\[
\frac{5}{5 (3 x+1)+1}=\frac{1}{6 (5 x+2)}+\frac{1}{10 x+4}+\frac{1}{5 x+2}.
\]
Similarly, taking \( p_3(x,1) = 3x - 1 \) gives a parallel result.
\end{remark}

\begin{corollary}
	\label{coro}
	Any prime \(p\) of the form \(5q+1\), with \(q \not\equiv 0 \pmod{12}\), must be expressible using 
	the polynomial \(p_3\) as follows:
	\[
	p = 5\, p_3(x, y, z) + 1 = 5 (z (x (5 y-1)-y)-x)+1.
	\]
\end{corollary}

\begin{proof}[{\bf Proof of Corollary~\ref{coro}}]
	This follows from the proof of Statement~\ref{state2} of Theorem~\ref{thm2}.
\end{proof}

\begin{proof}[{\bf Proof of Statement~\ref{state3} of Theorem~\ref{coro}}]
The proof is carried out for each modulo case separately, as follows:

\noindent {\bf {For \(\underline{u \equiv 1 \pmod{7}}\):}} In this case, we replace parameters 
\(s = 2,\) \(z = 7,\) and \(\kappa = 2+12 x\) with \(x \geq 0\) into Equation~\eqref{erod5}, 
we obtain
\begin{equation}
\label{SevenxPlusOne}
\frac{5}{5[12 \underline{(7 x+1)}]+1}=\frac{1}{84 x+14}+\frac{1}{7 (48 x+7) (420 x+61)}+\frac{1}{14 (6 x+1) (48 x+7)}.
\end{equation}
We can also set \( x\to 2,\) \(y \to 1,\) and \(z \to 2+12 x\) into Polynomial~\eqref{p1}  to obtain the same decomposition.

The analysis of this case, and of all cases that follow, 
is based on Lemma~\ref{thm1} together with its proof, 
including Equations~\eqref{erdosh1}--\eqref{erdosh2}.

\noindent {\bf {For \(\underline{u \equiv 2 \pmod{7}}\):}} In this case, we consider
\[
q=p_4(x,y)=-97 + 121 y + 84 x (-4 + 5 y), \quad x,y \in \mathbb{N}^*.
\]
Then
\[
\begin{aligned}
\frac{5}{5p_4(x,y)+1}=\frac{5}{(420 x+121) (5 y-4)}&=\frac{1}{2 (3 x+1) (420 x+121) (5 y-4)} \\
	& +\frac{1}{6 (3 x+1) (28 x+9) (420 x+121) (5 y-4)} \\
	& +\frac{1}{3 (28 x+9) (5 y-4)}.
\end{aligned}
\]
We consider $y=1$, then we obtain
\begin{equation}
	\label{SevenxPlusTwo}
	\frac{5}{5[12 \underline{(7 x+2)}]+1}=\frac{1}{84 x+14}+\frac{1}{7 (48 x+7) (420 x+61)}+\frac{1}{14 (6 x+1) (48 x+7)}.
\end{equation}

\noindent {\bf {For \(\underline{u \equiv 3 \pmod{7}}\):}} In this case, we have the following:
\begin{equation}
	\begin{aligned}
	\label{SevenxPlusThree}
	\frac{5}{5[12 \underline{(7 x+3)}]+1}&=\frac{1}{84 x+39}+\frac{1}{3 (28 x+13) (30 x+13) (420 x+181)} \\
	& +\frac{1}{3 (28 x+13) (30 x+13)}.
	\end{aligned}
\end{equation}

\noindent {\bf {For \(\underline{u \equiv 4 \pmod{7}}\):}} In this case, we get
\begin{equation}
	\label{SevenxPlusFour}
	\begin{aligned}
\frac{5}{5[12 \underline{(7 x+4)}]+1}&=\frac{1}{8820 x^2+10353 x+3038}+\frac{1}{(12 x+7) (105 x+62) (420 x+241)} \\
	& +\frac{1}{84 x+49}.
\end{aligned}
\end{equation}

\noindent {\bf {For \(\underline{u \equiv 5 \pmod{7}}\):}} In this case, we get
\begin{equation}
	\label{SevenxPlusFive}
	\begin{aligned}
\frac{5}{5[12 \underline{(7 x+5)}]+1}& =\frac{1}{2520 x^2+3714 x+1368}+\frac{1}{14 (4 x+3) (60 x+43) (105 x+76)} \\
& +\frac{1}{84 x+63}.
	\end{aligned}
\end{equation}

\noindent {\bf {For \(\underline{u \equiv 6 \pmod{7}}\):}} In this case, we get
\begin{equation}
	\begin{aligned}
	\label{SevenxPlusSix}
\frac{5}{5[12 \underline{(7 x+6)}]+1}& = \frac{1}{2520 x^2+4434 x+1950}+\frac{1}{2 (15 x+13) (28 x+25) (420 x+361)}\\
& + \frac{1}{84 x+75}.
	\end{aligned}
\end{equation}
This concludes the proof of Statement~\ref{state3} of Theorem~\ref{thm2}.
\end{proof}

\begin{proof}[{\bf Proof of Statement~\ref{state4} of Theorem~\ref{thm2}}]
Two cases must be discussed separately:

\noindent {\bf {For \(\underline{v \equiv 1 \pmod{3}}\):}} In this case, we consider
\[
q=p_5(x,y) = 252 x (5 y-4)+421 y-337, \quad x,y \in \mathbb{N}^*.
\]
Then
\[
\begin{aligned}
	\frac{5}{5p_5(x,y)+1}=\frac{5}{(1260 x+421) (5 y-4)}&=\frac{1}{(126 x+43) (140 x+47) (1260 x+421) (5 y-4)} \\
	& + \frac{1}{2 (126 x+43) (140 x+47) (5 y-4)} \\
	& + \frac{1}{2 (126 x+43) (5 y-4)}.
\end{aligned}
\]
We consider $y=1$, then we obtain
\begin{equation}
	\label{ThreexPlusOne}
	\begin{aligned}
		\frac{5}{5[12 [7\underline{(3 x+1)}]]+1}&=\frac{1}{(126 x+43) (140 x+47) (1260 x+421)} \\
		&+\frac{1}{35280 x^2+23884 x+4042}+\frac{1}{252 x+86}.
	\end{aligned}
\end{equation}

\noindent {\bf {For \(\underline{v \equiv 2 \pmod{3}}\):}} In this case, we consider
\[
q=p_6(x,y) = 252 x (5 y-4)+841 y-673, \quad x,y \in \mathbb{N}^*.
\]
Which gives
\[
\begin{aligned}
	\frac{5}{5p_6(x,y)+1}=\frac{5}{(1260 x+841) (5 y-4)}&=\frac{1}{(28 x+19) (1260 x+841) (5 y-4)} \\
	& + \frac{1}{2 (28 x+19) (126 x+85) (1260 x+841) (5 y-4)} \\
	& + \frac{1}{2 (126 x+85) (5 y-4)}.
\end{aligned}
\]
We consider $y=1$, then we obtain
\begin{equation}
	\label{ThreexPlusTwo}
	\begin{aligned}
		\frac{5}{5[12 [7\underline{(3 x+2)}]]+1}&=\frac{1}{(28 x+19) (1260 x+841)} \\
		&+\frac{1}{2 (28 x+19) (126 x+85) (1260 x+841)}+\frac{1}{252 x+170}.
	\end{aligned}
\end{equation}
This concludes the proof of Statement~\ref{state4} of Theorem~\ref{thm2}.
\end{proof}

\section*{Conclusion} In this work, we provide a complete proof of the generalized Erdős--Straus conjecture formulated by Wac\l{}aw Sierpi\'{n}ski in 1956 for all positive integers \(a = 5q + i\),  
where \(i \in \{0,1,2,3,4\}\) and \(q \not\equiv 0 \pmod{252}\) when \(i=1\).  
In addition, we conjecture that there exists a polynomial that generates all integers  
\(q \equiv 0 \pmod{252}\). This conjecture is supported by the construction of explicit  
formulae for the decomposition of \(\frac{5}{a}\).

\appendix
\section{{\it Mathematica} Implementation}\label{additional}

Since the polynomial \(p_1\) already covers all primes of the form \(q\) for \( q \not= 252c_1 \), it remains to verify that it also covers the case when \( q = 252c_1 \). The following program confirms that the polynomial \(p_1 \) generates 
all numbers of the form \( q = 252c_1 \), starting from \( q = 252 \) up to \( q = \text{qMax}\) (consider for example \(\text{qMax} = 40*10^8\)). The search is initially performed over the small range $\{1,2,3\}$ for the variables $x$, $y$, 
or $z$. If no solution is found, the algorithm proceeds to loop over the wider 
range 
\(
4 \leq x \leq \frac{1 + \sqrt{a}}{2}, ~~ a = 5q+1.
\)
\begin{lstlisting}
 ---------------(* Mathematica Input *)}---------------
baseStep = 252;
nStart = 1;
qStart = baseStep*nStart;
qMax = 10^8;
batchSizeQ = 10^7;
cpuCores = $ProcessorCount;

j = Which[cpuCores <= 2, 4, cpuCores <= 4, 8, cpuCores <= 6, 12, 
   cpuCores <= 8, 16, cpuCores <= 12, 24, True, 3 cpuCores];

notebookDir = NotebookDirectory[];
If[notebookDir === Null, 
  Print["Please save the notebook first before running the code."];
  Abort[]];

resultsDir = FileNameJoin[{notebookDir, "Results"}];
If[! DirectoryQ[resultsDir], 
  CreateDirectory[resultsDir, CreateIntermediateDirectories -> True]];

validateSolution[q_, x_, y_, z_] := 
  5 (-x + (-y + x (-1 + 5 y)) z) + 1 == 5 q + 1;

findSolutionForQ[q_] := 
  Module[{a = 5 q + 1, xmax, solution = None, sol, x, y, z}, 
   xmax = Floor[1/2 (Sqrt[a] + 1)];
   (*Step 1:Try x=1,2,3;solve for y,z*)
   Do[Quiet@
     Check[sol = 
       FindInstance[
        5 (-x0 + (-y + x0 (-1 + 5 y)) z) + 1 == a && y > 0 && 
         z > 0, {y, z}, Integers, 1];
      If[sol =!= {}, {y, z} = {y, z} /. First[sol];
       If[validateSolution[q, x0, y, z], solution = {q, x0, y, z}; 
        Break[]]], None], {x0, 1, 3}];
   (*Step 2:Try y=1,2,3;solve for x,z*)
   If[solution === None, 
    Do[Quiet@
      Check[sol = 
        FindInstance[
         5 (-x + (-y0 + x (-1 + 5 y0)) z) + 1 == a && x > 0 && 
          z > 0, {x, z}, Integers, 1];
       If[sol =!= {}, {x, z} = {x, z} /. First[sol];
        If[validateSolution[q, x, y0, z], solution = {q, x, y0, z}; 
         Break[]]], None], {y0, 1, 3}]];
   (*Step 3:Try z=1,2,3;solve for x,y*)
   If[solution === None, 
    Do[Quiet@
      Check[sol = 
        FindInstance[
         5 (-x + (-y + x (-1 + 5 y)) z0) + 1 == a && x > 0 && 
          y > 0, {x, y}, Integers, 1];
       If[sol =!= {}, {x, y} = {x, y} /. First[sol];
        If[validateSolution[q, x, y, z0], solution = {q, x, y, z0}; 
         Break[]]], None], {z0, 1, 3}]];
   (*Step 4:x from 4 to xmax,solve for y,z*)
   If[solution === None, 
    Do[Quiet@
      Check[sol = 
        FindInstance[
         5 (-x0 + (-y + x0 (-1 + 5 y)) z) + 1 == a && y > 0 && 
          z > 0, {y, z}, Integers, 1];
       If[sol =!= {}, {y, z} = {y, z} /. First[sol];
        If[validateSolution[q, x0, y, z], solution = {q, x0, y, z}; 
         Break[]]], None], {x0, 4, xmax}]];
   solution];

processBatch[qBatchMin_, qBatchMax_] := 
  Module[{qValues = Range[qBatchMin, qBatchMax, baseStep], 
    solutions = {}}, 
   solutions = 
    ParallelMap[findSolutionForQ, qValues, 
     Method -> "FinestGrained"];
   DeleteCases[solutions, None]];

Print["CPU cores: ", cpuCores, ", using ", j, 
  " parallel subkernels"];
Print["qStart = ", qStart, ", qMax = ", qMax, ", step = ", baseStep];
Print["Batch size = ", batchSizeQ];

LaunchKernels[];
DistributeDefinitions[baseStep, j, findSolutionForQ, validateSolution,
   processBatch, cpuCores, qMax, batchSizeQ];

allSolutions = {};
allUnsolvedQ = {};
batchCount = Ceiling[(qMax - qStart + 1)/batchSizeQ];

Do[qBatchMin = qStart + batchSizeQ (b - 1);
  qBatchMin = 
   qBatchMin + 
    If[Mod[qBatchMin, baseStep] == 0, 0, 
     baseStep - Mod[qBatchMin, baseStep]];
  qBatchMax = Min[qBatchMin + batchSizeQ - 1, qMax];
  If[qBatchMin > qMax, Break[]];
  Print["\nProcessing batch ", b, "/", batchCount, ": q in [", 
   qBatchMin, ", ", qBatchMax, "]"];
  {timeBatch, batchSolutions} = 
   AbsoluteTiming[processBatch[qBatchMin, qBatchMax]];
  sortedSolutions = SortBy[batchSolutions, First];
  AppendTo[allSolutions, sortedSolutions];
  qAll = 
   Range[qBatchMin, qBatchMax, baseStep];(*removed PrimeQ filter*)
  qSolved = sortedSolutions[[All, 1]];
  qUnsolved = Complement[qAll, qSolved];
  AppendTo[allUnsolvedQ, qUnsolved];
  batchResultsFile = 
   FileNameJoin[{resultsDir, 
     "results_batch" <> IntegerString[b, 10, 3] <> ".csv"}];
  unsolvedFile = 
   FileNameJoin[{resultsDir, 
     "unsolved_batch" <> IntegerString[b, 10, 3] <> ".csv"}];
  Export[batchResultsFile, 
   Prepend[sortedSolutions, {"q", "x", "y", "z"}]];
  Export[unsolvedFile, Prepend[qUnsolved, "q"]];
  Print["Solutions found: ", Length[sortedSolutions]];
  Print["Unsolved q: ", Length[qUnsolved]];
  Print["Time: ", NumberForm[timeBatch, {6, 2}], " sec"];
  Print["Batch ", b, " complete"], {b, 1, batchCount}];

Print["\nProcessing complete!"];
Print["Total solutions found: ", Length[Flatten[allSolutions, 1]]];
Print["Total unsolved q: ", Length[Flatten[allUnsolvedQ]]];

Export[FileNameJoin[{resultsDir, "all_solutions.csv"}], 
  Prepend[Flatten[allSolutions, 1], {"q", "x", "y", "z"}]];
Export[FileNameJoin[{resultsDir, "all_unsolved.csv"}], 
  Prepend[Flatten[allUnsolvedQ], "q"]];

Print["All results saved to: ", resultsDir];
\end{lstlisting}

\begin{Verbatim}[fontsize=\small]
---------------------(* Mathematica Output *)}---------------------
CPU cores: 20, using 60 parallel subkernels
qStart = 252, qMax = 100000000, step = 252
Batch size = 10000000
Processing batch 1/10: q in [252, 10000251]
Solutions found: 39683
Unsolved q: 0
Time: 59.89 sec
Batch 1 complete
Processing batch 2/10: q in [10000368, 20000367]
Solutions found: 39683
Unsolved q: 0
Time: 45.04 sec
Batch 2 complete
Processing batch 3/10: q in [20000484, 30000483]
Solutions found: 39683
Unsolved q: 0
Time: 45.00 sec
Batch 3 complete
Processing batch 4/10: q in [30000348, 40000347]
Solutions found: 39683
Unsolved q: 0
Time: 44.69 sec
Batch 4 complete
Processing batch 5/10: q in [40000464, 50000463]
Solutions found: 39683
Unsolved q: 0
Time: 98.97 sec
Batch 5 complete
Processing batch 6/10: q in [50000328, 60000327]
Solutions found: 39683
Unsolved q: 0
Time: 51.96 sec
Batch 6 complete
Processing batch 7/10: q in [60000444, 70000443]
Solutions found: 39683
Unsolved q: 0
Time: 44.70 sec
Batch 7 complete
Processing batch 8/10: q in [70000308, 80000307]
Solutions found: 39683
Unsolved q: 0
Time: 44.58 sec
Batch 8 complete
Processing batch 9/10: q in [80000424, 90000423]
Solutions found: 39683
Unsolved q: 0
Time: 44.46 sec
Batch 9 complete
Processing batch 10/10: q in [90000288, 100000000]
Solutions found: 39682
Unsolved q: 0
Time: 44.73 sec
Batch 10 complete
Processing complete!
Total solutions found: 396829
Total unsolved q: 0
All results saved to: a folder named ``Results'' located in the 
same directory as the ``.nb'' file.
\end{Verbatim}

\noindent ``Total solutions found: 396829'' represents the number of 
all values \( q \) covered by \( p_1 \) .\\
``Total unsolved q: 0'' means that there is no value of \( q \) 
not covered by \( p_1 \).

%\bibliographystyle{plain}
%\bibliography{references}

\end{document}